\documentclass[11pt,a4paper]{article}
\usepackage[english]{babel}

\usepackage{fullpage}

\usepackage{amsmath,amsfonts,amsthm}
\usepackage[colorlinks,linkcolor=blue]{hyperref}
\usepackage[abbrev,nobysame]{amsrefs}
\usepackage{tikz}
\usetikzlibrary{decorations.markings,intersections, calc}

\newtheorem{theorem}{Theorem}

\newtheorem{lemma}{Lemma}[section]
\newtheorem{cor}[lemma]{Corollary}

\theoremstyle{definition}
\newtheorem{definition}[lemma]{Definition}

\theoremstyle{remark}
\newtheorem{remark}[lemma]{Remark}

\newcommand{\R}{\mathbb R}
\newcommand{\ep}{\varepsilon}
\newcommand{\ga}{\gamma}
\newcommand{\pd}{\partial}
\newcommand{\co}{\colon}
\newcommand{\LL}{\mathcal L}
\newcommand{\ov}{\overline}

\newcommand{\Norm}{{\|{\cdot}\|}}

\newcommand\be{\begin{equation}}
\newcommand\ee{\end{equation}}
\numberwithin{equation}{section}

\begin{document}

\title{Local Finiteness of the Number of Reflections \\
in Semi-Dispersing Minkowski Billiards}

\author{R. V. Barinov$^{1,2}$, S. V. Ivanov$^{1,2}$}

\date{}

\maketitle

\footnotetext[1]{St.~Petersburg Department
of V.~A.~Steklov Mathematical Institute
of the Russian Academy of Sciences,
27 Fontanka, St. Petersburg, Russia, 191023.}

\footnotetext[2]{
St.~Petersburg State University,
7--9 Universitetskaya Emb., St.~Petersburg, Russia, 199034.}

\renewcommand{\thefootnote}{}
\footnotetext{This work is supported by the Russian Science Foundation
under project No~25-11-00058.}

\begin{abstract}
We prove that in a semi-dispersing Minkowski billiard, that is a billiard in the complement of a collection of convex sets in a normed space with a smooth strictly convex norm, any billiard trajectory of finite length has only finitely many reflections off the walls.
\end{abstract}

\section{Introduction}
In the 1970s, Ya.~G.\ Sinai formulated the following problem,
motivated by the study of ideal gas models:
in a system of $n$ hard balls in the Euclidean space $\R^d$
where the balls move uniformly and collide elastically,
is the number of collisions bounded?
The set of admissible positions of the balls in this problem
is naturally identified with the space $\R^{dn}$
from which one removes $\frac{n(n-1)}{2}$ cylindrical sets
corresponding to configurations where some pairs of balls overlap.
Upon an appropriate linear transformation,
the evolution of the system of balls can be interpreted
as a billiard trajectory in this configuration space
with the standard reflection rule
``angle of incidence equals angle of reflection'', see, e.g., \cite{Gal81}.

One can consider a more general class of billiards
where the ``billiard table'' is obtained by removing
a finite collection of convex sets with smooth boundaries
from~$\R^n$.
Such billiards are called \textit{semi-dispersing}.
In general, in a semi-dispersing billiard a trajectory may have
infinitely many reflections.
However it makes sense to ask about their local finiteness,
that is, finiteness within any bounded time interval.
Under appropriate non-degeneracy conditions ruling out ``zero angles''
between the walls of the billiard domain,
one may also consider the problem of a uniform bound on the local number of reflections over all trajectories, see \cite{Sinai78}.
Note that for billiards in convex domains, on the contrary, it is easy to construct examples where some trajectories reflect infinitely many times from the boundary within a bounded time interval, see, for example, \cite{Halpern77}.

The problem of local finiteness of reflections in semi-dispersing billiards
in $\R^n$ was solved affirmatively
by L. N. Vaserstein \cite{Vas79} and G. A. Galperin \cite{Gal81}.
In these works, the global finiteness of the number of collisions
along each trajectory in Sinai's hard balls problem was also proved.
However the question of a uniform bound over all trajectories
remained open for a long time.
Finally Burago, Ferleger, and Kononenko~\cite{BFK98},
using the geometry of Alexandrov spaces of non-positive curvature,
proved the uniform boundedness of both local and global number
of reflections in a semi-dispersing billiard with non-degenerate
intersections of walls.
As a corollary, the finiteness of the total number of collisions
in the hard balls problem was established.

The goal of this paper is to generalize the above mentioned results
on local finiteness of reflections to Minkowski billiards.
The difference from classical billiards is
that the ``billiard table'' is located not in $\R^n$
but in an arbitrary finite-dimensional normed space
with a smooth strictly convex norm.
In a normed space the standard billiard rule
``angle of incidence equals angle of reflection''
does not make sense.
It is replaced by a variational formulation:
a billiard trajectory realizes a local extremum of the arc-length
functional with free points on the boundary of the billiard domain.
Formal definitions are given in \S\ref{sec:prelim}, see also \cite{GT02}.

We also note that, in the case of a convex billiard domain,
there is a correspondence between Euclidean billiards and
projective Minkowski billiards,
in the sense that if a Minkowski billiard in a convex domain is projective
then it is a usual Euclidean billiard under a suitable Euclidean structure,
see~\cite{GluMat25}.
General information about projective billiards can be found in~\cite{ProjTab}.

Like in the Euclidean case, a Minkowski billiard
is called semi-dispersing if the complement
of the billiard domain can be represented as a union
of a finite collection of convex sets (``walls'') with smooth boundaries.
We do not rule out trajectories that ``hit a corner''
but regard them defined only up to the moment of such a hit,
see definitions in \S\ref{sec:prelim}.
The main result of this paper is the following.

\begin{theorem}\label{t:time}
In a semi-dispersing Minkowski billiard, the set of reflection moments
along any billiard trajectory of finite length is finite.
\end{theorem}

The theorem implies that the only obstruction to an infinite extension
of a billiard trajectory is a ``direct hit in a corner''
after a finite number of reflections, see Corollary~\ref{cor:extension}.

Note that, in general, there is no uniform bound
on the local number of reflections in Theorem~\ref{t:time} without additional assumptions.
For example (see \cite{Gal81}*{\S5}),
consider a billiard in the complement of the union of two tangent round discs
in the Euclidean plane
and trajectories with initial data close to the common tangent line of the circles.
For each of these trajectories the number of reflections is finite
but this number can be made arbitrarily large by choosing suitable initial data.

In view of such examples one naturally asks whether
a billiard trajectory can accumulate infinite length
and infinitely many reflections while remaining
in a small neighbourhood of a single point.
The following theorem shows that this is impossible.

\begin{theorem}\label{t:space}
In a semi-dispersing Minkowski billiard,
for any point $p$ in the closure of the billiard domain and any $\ep>0$,
there exists a neighborhood $U\ni p$ such that
any billiard trajectory entirely contained in $U$ has length at most $\ep$.
\end{theorem}

The proofs of Theorems~\ref{t:time} and \ref{t:space}
are given in \S\ref{sec:proofs}.
They are based on a special semi-invariant,
which we call the \textit{distance momentum}.
The definition of the distance momentum and the proof
of its key properties are given in \S\ref{sec:momentum}.
The preliminary section \S\ref{sec:prelim} introduces notation
and discusses general properties of Minkowski billiards.

\paragraph{Motivations.}
As mentioned above, the best results on the number of reflections
in Euclidean semi-dispersing billiards
have been obtained using geometry of Alexandrov spaces,
see \cites{BFK98}, \cite{BFK98b}, and~\cite{AKP19}.
They are based on the idea that a billiard trajectory is a geodesic
in a metric space obtained by gluing copies of~$\R^n$
along convex sets corresponding to the billiard walls.
By Reshetnyak's gluing theorem, the resulting space
has non-positive Alexandrov curvature, that is, it belongs to the CAT(0) class.
Distance inequalities in CAT(0) spaces play a key role in estimating
the number of reflections.

Alexandrov spaces can be regarded as a non-smooth generalization
of Riemannian manifolds.
An interesting problem is to find a similar generalisation for Finsler manifolds.
(Finsler manifolds differ from Riemannian ones
in that the norms on tangent spaces are not assumed Euclidean).
The definition of bounded curvature in the Alexandrov sense
does not apply to Finsler metrics;
no normed spaces other than Euclidean ones satisfy it.
A similar notion of a space of non-positive curvature
in the sense of Busemann \cite{Bus48}
includes all normed spaces with strictly convex norms.
However for Finsler metrics, the Busemann non-positive
curvature condition is also too restrictive, see \cite{IL19}.
For example in the two-dimensional case,
among non-Riemannian Finsler metrics only flat metrics satisfy this condition.
Moreover, Reshetnyak's gluing theorem does not hold for Busemann non-positive curvature, see \cite{BI13}.

On the other hand, for polyhedral Finsler spaces
(glued from simplices of normed spaces along isometries of faces),
a globalization theorem for uniqueness and minimality of geodesics holds,
see~\cite{BI13}.
Trajectories of semi-dispersing billiards are also related
to the globalization of minimality of geodesic
in a suitable space (compare Corollary~\ref{cor:hinge comparison} below with Comparison Lemma 2.1 in \cite{BFK98}).
These observations indicate that there might be a natural class of spaces
including CAT(0)-spaces and normed spaces with strictly convex norms
and featuring the global minimality of geodesics
and an analogue of Reshetnyak's gluing theorem.

\section{Minkowski Billiards}
\label{sec:prelim}

In this section we introduce notation and collect facts
required for the formulation of the main theorems and the
reflection law in semi-dispersing Minkowski billiards.
A detailed discussion of Minkowski billiards can be found in~\cite{GT02}.

\subsection{Definitions and Notation}\label{subsec:prelim}

\begin{definition} \label{d:minkowki norm}
A \textit{Minkowski space} is a finite-dimensional real vector space $X$
endowed with a \textit{Minkowski norm}, that is, a continuous function
$\Norm\co X\to\R$ satisfying the following conditions:

\begin{enumerate}

\item
Positive definiteness: $\|x\|>0$ for all $x\in X\setminus\{0\}$.
\item
Positive homogeneity: $\|tx\|=t\|x\|$ for all $x\in X$ and $t\ge 0$.
\item
Subadditivity: $\|x+y\|\le \|x\|+\|y\|$ for all $x,y\in X$.  
Recall that for positively homogeneous functions, subadditivity is equivalent to convexity.
\item
Strict convexity, in the sense that the set $\{x\in X: \|x\|=1\}$ (called the \textit{indicatrix} or the \textit{unit sphere} of the norm) contains no line segments. This condition is equivalent to the subadditivity condition being strict for linearly independent vectors $x,y$.
\item
$C^1$-smoothness on $X\setminus\{0\}$.

\end{enumerate}

\end{definition}

For the sake of brevity, the Minkowski norms will be referred to simply as norms.
Note that the Minkowski norm is not assumed symmetric:
in general, $\|x\|\ne\|{-x}\|$.
Non-symmetric norms are common in Finsler geometry
where our motivations come from.
The lack of symmetry does not affect the arguments
except that one has to be careful with some formulas.
For example, the distance $\|x-y\|$ between points $x$ and $y$
is not symmetric in $x$ and~$y$.
We denote by $\Norm^-$ the opposite norm,
defined by $\|x\|^-=\|{-x}\|$ for all $x\in X$.

In the literature, Minkowski norms are usually assumed
to be $C^\infty$-smooth and quadratically strictly convex,
but for the purposes of this paper the above conditions are sufficient.
Everywhere in the sequel the word ``smooth'' means $C^1$.

A \textit{Minkowski billiard} is determined by
a Minkowski space $X=(X^n,\Norm)$ and a domain $\Omega\subset X$,
called the \textit{billiard table}.
We denote by $\ov\Omega$ the closure of $\Omega$,
and by $\pd\Omega$ its boundary.
A point $p\in\pd\Omega$ is called \textit{smooth} if,
in a neighborhood of $p$, the boundary $\pd\Omega$ is a smooth hypersurface.
For a smooth point $p\in\pd\Omega$,
we denote by $T_p\pd\Omega$ the hyperplane in $X$
containing $0$ and parallel to the tangent hyperplane to $\pd\Omega$ at $p$.

A billiard trajectory can be informally described
as a path of a particle that moves uniformly inside $\Omega$
and when it hits a smooth point $p\in\pd\Omega$,
it reflects from the boundary according to
the following variational principle:
if $q_1,q_2$ are sufficiently close points of the trajectory
before and after the reflection,
then $p$ is a critical point of the function
\be\label{e:local length}
  x \mapsto \|x-q_1\| + \|q_2-x\|, \qquad x\in \pd\Omega.
\ee
We assume that trajectories are parametrized with unit speed.
The formal details are given in the following definition.

\begin{definition}\label{d:trajectory}
The \textit{trajectory} of the Minkowski billiard $(X,\Norm,\Omega)$
is a continuous curve $\ga\co I\to\ov\Omega$,
parametrized by an interval $I\subset\R$,
such that for any interior point $t\in I$ there exists $\delta>0$ such that,
for all $t_1\in(t-\delta,t)$ and $t_2\in(t,t+\delta)$
at least one of the following conditions holds:
\begin{enumerate}
\item[(1)]
The restriction $\ga|_{[t_1,t_2]}$ parametrizes
a segment of an affine line in $X$
with unit speed (i.e.\ $\|\dot\ga\|=1$).
\item[(2)]
The point $p=\ga(t)$ is a smooth boundary point of $\pd\Omega$,
the restrictions $\ga|_{[t_1,t]}$ and $\ga|_{[t,t_2]}$
are unit-speed parametrizations of the segments $[q_1,p]$ and $[p,q_2]$
where $q_1=\ga(t_1)$ and $q_2=\ga(t_2)$,
and $p$ is a critical point of \eqref{e:local length} on $\pd\Omega$.
\end{enumerate}
The values of $t$ for which only the second condition
holds are called \textit{reflection moments},
and the corresponding points $\ga(t)$ are called \textit{reflection points}.
The \textit{length} of a trajectory $\ga$ is the length of the interval~$I$.

\end{definition}

We emphasize that the conditions in this definition apply
only the interior points of the interval $I$;
we do not impose any conditions on
the behavior of the trajectory near its endpoints.
One may assume that the trajectories are parametrized by open intervals
and the endpoint values are defined by continuity if necessary.

\begin{definition}
A Minkowski billiard is called \textit{semi-dispersing}
if the billiard table has the form
\be\label{e:walls}
 \Omega = X \setminus \bigcup_{i=1}^m W_i,
\ee
where $W_1,\dots,W_m\subset X$ are closed convex sets with nonempty interiors and smooth boundaries.
The sets $W_i$ are referred to as the \textit{walls}.
\end{definition}

\subsection{Reflection Law}\label{subsec:legendre}

To describe reflections in a more convenient form
we need the Legendre transform of the norm,
see~\cite{GT02}*{\S3}.
Let $(X,\Norm)$ be a Minkowski space.
Consider the dual space $X^*$,
it carries the \textit{dual norm} $\Norm^*$ defined by
\be\label{e:dual norm}
 \|f\|^* = \sup \{ f \cdot x : x\in X,\, \|x\|\le 1 \},
 \qquad f\in X^*,
\ee
where the dot in $f \cdot x$ denotes the application of a co-vector to a vector.
The dual norm is also positively 1-homogeneous and subadditive
but not necessarily symmetric.
One can check that $\Norm^*$ is also a Minkowski norm in the sense
of Definition~\ref{d:minkowki norm} but we do not use this fact.

The \textit{Legendre transform} of the norm $\Norm$
is a positively 1-homogeneous map $\LL=\LL_{\Norm}\co X\to X^*$ defined by
\be\label{e:legendre}
 \LL(x) = d_x\big(\tfrac12\Norm^2\big) = \|x\|\, d_x(\Norm), \qquad x\in X,
\ee
where $d_x$ denotes the differential of a function at~$x$.
In the Lagrangian mechanics language,
$\LL$ is the Legendre transform mapping for the Lagrangian $\tfrac12\Norm^2$.
We will apply $\LL$ to unit vectors only.
We need the following simple properties.

\begin{lemma}\label{l:legendre basic}
Let $(X,\Norm)$ be a Minkowski space,
$v\in X$, $\|v\|=1$, and $f=\LL_\Norm(v)$.
Then
\begin{enumerate}
 \item
 $\|f\|^*=1$ and $f\cdot v=1$.
 In particular, the supremum in \eqref{e:dual norm}
 is attained at $x=v$.
 \item
 The vector $v$ is uniquely determined by the equalities $\|v\|=1$ and $f\cdot v=1$. That is, if $w\in X$ satisfies $\|w\|=1$ and $f\cdot w=1$, then $w=v$.
\end{enumerate}
\end{lemma}

\begin{proof}
1. Let $x\in X$ and $\|x\|\le 1$.
By the subadditivity of the norm,
for every $t>0$ we have \(\|v+tx\| \le \|v\|+t\|x\|\le 1+t\),
hence
\[
 f\cdot x = d_v(\Norm)\cdot x = \frac d{dt}\Big|_{t=0} \|v + tx\| \le 1 .
\]
Therefore $\|f\|^*\le 1$. At $x=v$, all inequalities above turn into equalities,
hence $f\cdot v = 1$ and $\|f\|^*=1$.

2. Suppose the contrary.
Then there exists $w\in X$ such that $\|w\|=1$ and $f\cdot w=1$ but $w\ne v$.
We are going to show that the segment $[v,w]$ is contained
in the unit sphere of~$\Norm$.
Pick $x\in[v,w]$.
The equalities $f\cdot v=f\cdot w=1$ imply that $f\cdot x=1$,
hence $\|x\|\ge 1$ since $\|f\|^*=1$.
On the other hand, $\|x\|\le 1$ since
$\|v\|=\|w\|=1$ and the norm is subadditive.
Thus $\|x\|=1$. Since $x$ was arbitrary on the segment $[v,w]$,
we conclude that this segment in contained in the unit sphere of the norm.
This contradicts the strict convexity of the norm
from Definition~\ref{d:minkowki norm}.
\end{proof}

Geometrically, the first part of Lemma~\ref{l:legendre basic} means that
the indicatrix hyperplane $f^{-1}(1)$
is tangent to the unit sphere of $\Norm$ at~$v$,
while the second part asserts uniqueness of this touch point.
This implies that $\LL$ is a homeomorphism
between the unit spheres of $\Norm$ and $\Norm^*$.

The following lemma reformulates the reflection law of Minkowski billiards.
In the sequel this reformulation will be used instead of the definition.

\begin{lemma}\label{l:reflection law}
Let $(X,\Norm,\Omega)$ be a Minkowski billiard,
$p\in\ov\Omega$ a smooth boundary point,
and $\ga$ a two-segment polygonal chain in $\ov\Omega$ with a break at~$p$, parametrized with unit speed.
Let $v_-$ and $v_+$ be the velocity vectors of $\ga$ before and after the break.

Then $\ga$ is a billiard trajectory if and only if
\be\label{e:reflection}
 (\LL(v_+)-\LL(v_-))|_{T_p\pd\Omega} = 0.
\ee
\end{lemma}

\begin{proof}
The statement easily follows from the arguments in \cite{GT02}*{\S3};
see in particular \cite{GT02}*{Corollary 3.2}.
For completeness we give a proof consistent with the
above definitions and notation.

Let $q_1$ and $q_2$ be points on $\ga$ before and after the break.
The second equality in \eqref{e:legendre} implies
that the differential of the norm at any point equals
the Legendre transform of the corresponding unit vector:
\[
 d_x\Norm = \LL(x/\|x\|), \qquad x\in X\setminus\{0\}.
\]
Therefore, the differentials of the two terms in \eqref{e:local length} at~$p$
are $\LL(v_-)$ and $-\LL(v_+)$, respectively.
Thus \eqref{e:reflection} is equivalent to $p$ being a critical point
of the function \eqref{e:local length} on $\pd\Omega$.

This proves the lemma in the case when $\ga$ has a genuine break at $p$.
If it degenerates into a straight segment,
then $v_-=v_+$ and \eqref{e:reflection} is trivial.
\end{proof}

\begin{remark}\label{rem:not tangent}
Lemma~\ref{l:reflection law} implies that at the reflection moment neither of the adjacent trajectory segments is tangent to the boundary. Indeed, suppose for instance that $v_-\in T_p\pd\Omega$.
Then by\eqref{e:reflection} we have
\(
(\LL(v_+)-\LL(v_-))\cdot v_- = 0,
\)
hence
\[
 \LL(v_+)\cdot v_- = \LL(v_-)\cdot v_- = 1,
\]
which implies $v_+=v_-$, since $v_+$ is the unique unit vector $v$ such that $\LL(v_+)\cdot v=1$ (see Lemma~\ref{l:legendre basic}).
The case when $v_+\in T_p\pd\Omega$ is similar.
\end{remark}

\subsection{Extension of trajectories}
The following lemma address the question
of the extendability of the billiard trajectory after hitting the boundary.
It is not used in the proofs of Theorems~\ref{t:time} and~\ref{t:space}
but confirms that our definitions make sense.
We could not find this statement in the literature
so we give a detailed proof here.

\begin{lemma}\label{l:reflection exists}
Let $X=(X,\Norm)$ be a Minkowski space,
$H\subset X$ a hyperplane containing~$0$,
$v_-\in X\setminus H$ and $\|v_-\|=1$.
Then there exists a unique vector $v_+\in X$ such that
$v_+\ne v_-$, $\|v_+\|=1$, and
\be\label{e:reflection in H}
(\LL(v_+)-\LL(v_-))|_H = 0,
\ee
where $\LL$ is the Legendre transform of the norm $\Norm$.
Furthermore $v_+$ and $v_-$ belong to different open half-spaces
bounded by~$H$.
\end{lemma}

Lemma \ref{l:reflection exists} ``is geometrically obvious''
from the following (see Fig.~\ref{fig:ellipse}).
Consider the indicatrix $G_-$ of~$\LL(v_-)$.
This is an affine hyperplane which
(except the case of parallelism) intersects $H$
in some affine subspace $Y$ of dimension $n-2$ where $n=\dim X$.
The subspace $Y$ does not intersect the unit ball $S$ of the norm~$\Norm$.
The relation \eqref{e:reflection in H} is equivalent
to the property that the indicatrix of $\LL(v_+)$ also contains~$Y$.
It is clear that among the affine hyperplanes containing~$Y$
there are exactly two tangent to the unit sphere of the norm.
The second of them, denoted $G_+$, corresponds to the desired vector~$v_+$.

Below is a more formal proof, which however lacks geometric clarity.

\begin{figure}[h!]
\centering
\begin{tikzpicture}
[
    tangent/.style={
        decoration={
            markings,
            mark=at position #1 with {
                \coordinate (tangent point-\pgfkeysvalueof{/pgf/decoration/mark info/sequence number}) at (0pt,0pt);
                \coordinate (tangent unit vector-\pgfkeysvalueof{/pgf/decoration/mark info/sequence number}) at (1,0pt);
                \coordinate (tangent orthogonal unit vector-\pgfkeysvalueof{/pgf/decoration/mark info/sequence number}) at (0pt,1);
            }
        },
        postaction=decorate
    },
    use tangent/.style={
        shift=(tangent point-#1),
        x=(tangent unit vector-#1),
        y=(tangent orthogonal unit vector-#1)
    },
    use tangent/.default=1
]

\draw[
    tangent=0.15,
    tangent=0.5,
] (0,0) ellipse [x radius=3.6cm, y radius=2cm, rotate=45]
node at (1.35,-1.8) {$S$};

\draw[ use tangent, name path=tanA] (-0.75,0) -- (6.3,0) node [pos = 0.6, above] {$G_+$};

\draw[ use tangent=2, name path=tanB] (-3.55,0) -- (0.5,0) node [midway, below] {$G_-$};

\path[name intersections={of=tanA and tanB, by=I}];

\draw[name path=hor] ($(I)+(-0.5,0)$) -- ($(I)+(8.5,0)$) node[above] {$H$};
\fill[] (I) circle (0.5pt) node[above left] {$Y$};

\fill[] (tangent point-1) circle (0.5pt);
\fill[] (tangent point-2) circle (0.5pt);

\coordinate (O) at ($(I)+(6,0)$);
\fill[] (O) circle (0.5pt) node[above right] {$0$};

\draw[->, >=latex] (O) -- (tangent point-1) node[midway, above right] {$v_{+}$};
\draw[->, >=latex] (O) -- (tangent point-2) node[midway, below right] {$v_{-}$};
    
\end{tikzpicture}
\caption{Uniqueness of $v_+$ for a given $v_-$}
\label{fig:ellipse}
\end{figure}
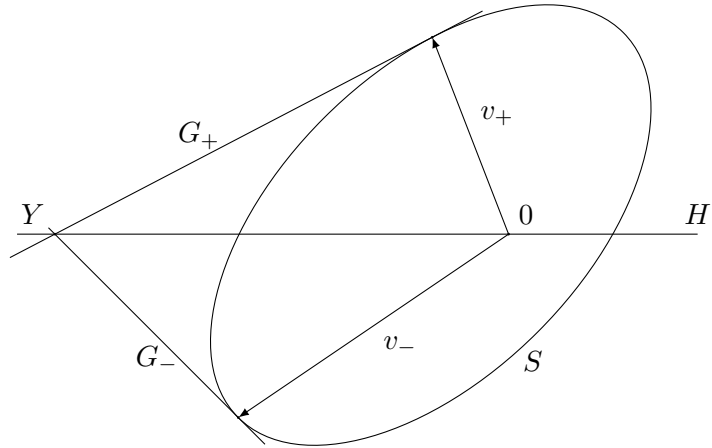

\begin{proof}[Proof of Lemma \ref{l:reflection exists}]
Set $f_-=\LL(v_-)$ and pick a co-vector $f_0\in X^*$
such that $\ker f_0=H$ and $f_0(v_-)>0$.
Consider a straight line $\ell\subset X^*$ given by
\[
 \ell = \{ f_- + t f_0 : t\in \R \}.
\]
The conditions on $v_+$ are equivalent to requiring that $\LL(v_+)$ has unit norm, differs from $f_-$, and lies on $\ell$. Define
\[
 \varphi(t) = \|f_- + t f_0\|^*  
\]
for all \(t\in \R\).
The function $\varphi$ is differentiable everywhere
except possibly at one point where it vanishes.
Note that $\varphi(0)=f_-(v_-)=1$ and $\varphi'(0)>0$ since
$$
 \varphi(t) \ge (f_- + t f_0)(v_-) = 1 + t f_0(v_-)
$$
for all $t>0$,
therefore $\varphi'(0)\ge f_0(v_-)>0$.
Moreover $\varphi$ is a convex function and $\varphi(t)\to +\infty$ as $|t|\to\infty$.
These properties imply that the equation $\varphi(t)=1$ has exactly two roots,
$t=0$ and $t=\alpha$ for some $\alpha<0$.
Thus $\ell$ contains a unique unit co-vector $f_+=f_-+\alpha f_0$
distinct from $f_-$.
The desired vector $v_+$ is obtained from $f_+$ by the inverse Legendre transform.

It remains to check that $v_+$ and $v_-$
are located in different open half-spaces bounded by~$H$.
Assume the contrary, then $f_0(v_+)\ge 0$, hence
\[
 f_-(v_+) = f_+(v_+) - \alpha f_0(v_+) \ge 1 ,
\]
since $\alpha<0$.
This contradicts the uniqueness of $v_-$ as the vector
realizing the maximum of $f_-$ on the unit sphere of the norm
(see Lemma~\ref{l:legendre basic}).
\end{proof}

Lemmas~\ref{l:reflection law} and~\ref{l:reflection exists} imply that
any segment in $\ov\Omega$ ending at a smooth boundary point $p\in\pd\Omega$
and not tangent to the boundary,
has a unique local extension satisfying the definition
of a Minkowski billiard trajectory.
Thus any initial segment $\ga\co[0,\ep]\to\ov\Omega$ has a unique
maximal extension forward in time.
It is defined either on $[0,+\infty)$ or on some bounded interval $[0,T]$.
In the latter case it ends in one of the following three ways
(in all cases $\ga(T)\in\pd\Omega$):
\begin{enumerate}
\item[(1)] The trajectory ends with a segment
ending at a non-smooth point of the boundary.
\item[(2)] the trajectory ends with a segment tangent to the boundary at $\ga(T)$,
and no extension of this segment is contained in $\ov\Omega$;
\item[(3)] When approaching the right endpoint of $[0,T]$,
the trajectory encounters infinitely many reflections from the walls.
\end{enumerate}

For semi-dispersing billiards, the second case is clearly impossible.
Theorem~\ref{t:time} also rules out the third case.
Thus Theorem~\ref{t:time} implies the following

\begin{cor}\label{cor:extension}
In a semi-dispersing Minkowski billiard,
any maximal forward trajectory having a starting point
either has infinite length
or is a finite polygonal chain ending at a non-smooth boundary point.
\end{cor}

Similar properties hold for backward extensions.
They can be reduced to forward ones by reversing the time
and replacing the Minkowski norm with its opposite.

\section{Distance momentum} \label{sec:momentum} 

Consider a semi-dispersing Minkowski billiard $(X,\Norm,\Omega)$ of the form \eqref{e:walls}.
The collection of walls $\{W_i\}_{i=1}^m$ is fixed.
Let $\ga\co I\to\ov\Omega$ be a billiard trajectory.
These notations are fixed throughout the rest of the paper.

\begin{definition}\label{d:momentum}
Let $p\in X$ and let $t\in I$ be a moment that
is neither a moment of reflection nor the endpoint of~$I$.
The \textit{distance momentum} of~$\ga$
with respect to~$p$ at~$t$ is defined by
\[
M_{\ga,p}(t) = \LL(\dot\ga(t))\cdot(\ga(t)-p) , 
\]
where $\LL$ is the Legendre transform of~$\Norm$, see~\S\ref{subsec:legendre}.
\end{definition}

\begin{remark}\label{rem:euclidean case}
If the norm $\Norm$ is Euclidean,
the definition of distance momentum can be rewritten as
\[ 
M_{\ga,p}(t) = \langle\dot\ga(t) , \ga(t)-p \rangle = \frac12 \frac d{dt} \|\ga(t)-p\|^2 , 
\]
where $\langle,\rangle$ is the scalar product associated with the norm.
For an arbitrary norm the distance momentum
does not have such a relation to the distance function to~$p$
along the trajectory
(and the distance function does not have the properties we need).
\end{remark}

\begin{lemma}\label{l:momentum bounds}
Let $\ga$, $p$, $t$ be as in Definition~\ref{d:momentum}. Then 
\[
-\|p-\ga(t)\| \le  M_{\ga,p}(t) \le \|\ga(t)-p\| . 
\]
The first inequality turns into equality
if and only if the vector $\ga(t)-p$
is proportional to $\dot\ga(t)$ with a non-positive coefficient,
and the second one turns into equality
if and only if $\ga(t)-p$ is proportional to $\dot\ga(t)$
with a non-negative coefficient.
\end{lemma}

\begin{proof} 
Let $f=\LL(\dot\ga(t))$.
Then $\|f\|^*=1$ and $\dot\ga(t)$ is the unique unit vector
where $f$ attains the value~1, see Lemma~\ref{l:legendre basic}.
Therefore, for any $w\in X$ the inequality $f\cdot w\le \|w\|$ holds, with equality if and only if $w$ is proportional to $\dot\ga(t)$ with a non-negative coefficient. Substituting $w=\ga(t)-p$ we get the second inequality of the lemma and its equality condition. Substituting $w=p-\ga(t)$ we get $-M_{\ga,p}(t)\le \|p-\ga(t)\|$, with equality if and only if $p-\ga(t)$ is proportional to $\dot\ga(t)$ with a non-negative coefficient. This proves the lemma.
\end{proof}

The next lemma is a key step in the proofs
of Theorems~\ref{t:time} and~\ref{t:space}.
In the Euclidean case (see Remark~\ref{rem:euclidean case})
its statement is equivalent to 2-convexity of the squared
distance to~$p$ along the trajectory,
i.e.\ convexity of the function $t\mapsto \|\ga(t)-p\|^2-t^2$.
For an arbitrary norm the distance function to~$p$ along the trajectory
may fail to be convex,
and such examples can already be constructed
for Minkowski billiards in a half-plane.

\begin{lemma}\label{l:monotone}
Let $p\in X$. Suppose that $\ga$ reflects only from those walls $W_i$
that contain~$p$.
Then for any $t_1,t_2\in I$ distinct from reflection moments
and interval endpoints, with $t_1<t_2$, one has
\be\label{e:monotone} 
M_{\ga,p}(t_2) - M_{\ga,p}(t_1)  \ge t_2 - t_1 . 
\ee
\end{lemma}

\begin{proof}
The lemma is equivalent to the statement that the function $\varphi$ defined by 
\[
\varphi(t) = M_{\ga,p}(t) - t 
\]
is non-decreasing on its domain.
It suffices to verify this property on every linear segment of the trajectory
and in a neighborhood of every reflection moment.
First, consider the case when there are no reflection moments
between $t_1$ and $t_2$, i.e.\ when $\ga|_{[t_1,t_2]}$ is a straight line segment.
Let $v$ be the velocity vector of $\ga$ on $[t_1,t_2]$.
The equalities $\dot\ga(t_1)=\dot\ga(t_2)=v$
and $\ga(t_2) - \ga(t_1) = (t_2-t_1) v$
imply that
\[
M_{\ga,p}(t_2) - M_{\ga,p}(t_1) = \LL(v)\cdot \big((\ga(t_2)-p)- (\ga(t_1)-p)\big) = (t_2-t_1) \LL(v)\cdot v = t_2-t_1 , 
\]
since $\LL(v)\cdot v=1$. Hence $\varphi(t_1)=\varphi(t_2)$.
Thus $\varphi$ is constant on each linear segment of the trajectory.

Now suppose that there is exactly one reflection moment $t_0$
between $t_1$ and $t_2$.
Let $q=\ga(t_0)$ be the reflection point
and $v_-,v_+$ the velocity vectors before and after the reflection.
As shown above, $\varphi$ is constant on $[t_1,t_0)$ and $(t_0,t_2]$.
Passing to the limit in the definition of the distance momentum, we obtain
\[
\varphi(t_1) = \lim_{t\uparrow t_0} \varphi(t) = \LL(v_-)\cdot(q-p) - t_0 
\]
and 
\[
\varphi(t_2) = \lim_{t\downarrow t_0} \varphi(t) = \LL(v_+)\cdot(q-p) - t_0 , 
\]
therefore
\be\label{e:monotone1}
\varphi(t_2) - \varphi(t_1) = (\LL(v_+)-\LL(v_-))\cdot (q-p) . 
\ee
Our goal is to show that the right-hand side of \eqref{e:monotone1} is non-negative.
Define $f=\LL(v_+)-\LL(v_-)$.
Note that $f\ne 0$ since $v_+\ne v_-$. Let $H=T_q\pd\Omega$, then by Lemma~\ref{l:reflection law} we have $\ker f=H$.

Let $H_-$ and $H_+$ denote the open half-spaces bounded by~$H$
and containing tangent vectors to $X$ at $q$
pointing outwards and inwards $\Omega$, respectively.
Observe that $v_-\in H_-$ and $v_+\in H_+$ (see Remark~\ref{rem:not tangent}).

Since $\ker f=H$, the sign of $f$ is constant on each
of the half-spaces $H_-$ and $H_+$.
Moreover
\[
f\cdot v_+ = (\LL(v_+)- \LL(v_-))\cdot v_+ > 0, 
\]
because $\LL(v_+)\cdot v_+ = 1$ and $\LL(v_-)\cdot v_+ < 1$.
Hence $f\cdot x>0$ for all $x\in H_+$ and $f\cdot x<0$ for all $x\in H_-$.

Let $W_k$ be the billiard wall containing $q$.
By the assumption of the lemma we have $p\in W_k$.
Since $W_k$ is convex, the affine hyperplane $q+H$
is a supporting hyperplane for $W_k$ at~$q$.
The definition of $H_-$ and $H_+$ implies that
$W_k$ lies on the ``negative'' side of this hyperplane,
i.e.\ $W_k\subset q+\ov H_-$.
In particular, $p\in q+\ov H_-$,
hence $p-q\in \ov H_-$ and $q-p\in \ov H_+$ (see Fig.~\ref{fig:q-p}).
The latter and the positivity of $f$ on $H_+$
imply that $f\cdot(q-p) \ge 0$.
Taking into account \eqref{e:monotone1}
we obtain that $\varphi(t_2)\ge\varphi(t_1)$
and the lemma follows.
\end{proof}

\begin{figure}[h!]
\centering
\begin{tikzpicture}[
    tangent/.style={
        decoration={
            markings,
            mark=
                at position #1
                with
                {
                    \coordinate (tangent point-\pgfkeysvalueof{/pgf/decoration/mark info/sequence number}) at (0pt,0pt);
                    \coordinate (tangent unit vector-\pgfkeysvalueof{/pgf/decoration/mark info/sequence number}) at (1,0pt);
                    \coordinate (tangent orthogonal unit vector-\pgfkeysvalueof{/pgf/decoration/mark info/sequence number}) at (0pt,1);
                }
        },
        postaction=decorate
    },
    use tangent/.style={
        shift=(tangent point-#1),
        x=(tangent unit vector-#1),
        y=(tangent orthogonal unit vector-#1)
    },
    use tangent/.default=1
]
\draw[
    tangent=0.1,
    tangent=0.635
] (0,0)
    to [out=30,in=100] (10,0) node at (9.35,0.75) {$W_k$};

\fill[] (tangent point-1) circle (1pt) node[above] {$p$};
\fill[] (tangent point-2) circle (1pt) node[below] {$q$};

\draw[ use tangent=2, name path=H] (-3,0) -- (3,0) node[above left] {$q+H$};

\draw[dashed, use tangent=2, shift=(tangent point-1)] (-1.5,0) -- (1.5,0) node[pos=0.8,below] {$p+H$};

\draw[<-, >=latex, use tangent = 2] (tangent point-2) -- (1, 1) node[pos = 0.7, left ] {$v_-$};
\draw[<-, >=latex, use tangent = 2] (-1.5, 0.75) -- (tangent point-2) node[pos = 0.5, above ] {$v_+$};

\draw[dashed, <-, >=latex, use tangent=2, shift=(tangent point-1)] (-1.5, 0.75) -- (tangent point-1) node[pos = 0.5, above ] {$v_+$};

\draw[ ->, >=latex] (tangent point-1) -- (tangent point-2) node[midway, below] {$q-p$};

\end{tikzpicture}
\caption{The vectors $v_+$ and $q-p$ belong to the half-space $H_+$}
\label{fig:q-p}
\end{figure}
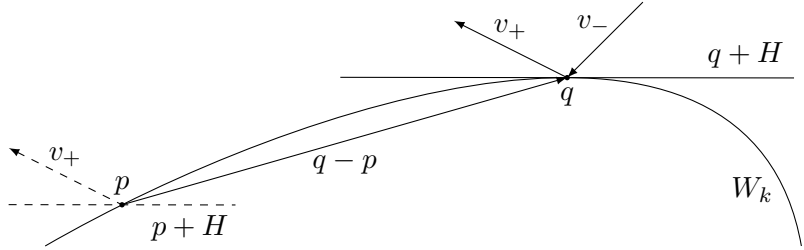

As a consequence we obtain an analogue of the Comparison Lemma from~\cite{BFK98}
(see also~\cite[Lemma 2.1]{BFK98}),
which states that the length of a billiard trajectory
contained in a sufficiently small neighborhood of a point~$p$
does not exceed the sum of the distances from $p$
to the endpoints of the trajectory.

\begin{cor}\label{cor:hinge comparison}
Let $p\in X$, and suppose that $\ga$ reflects only from those walls $W_i$
that contain~$p$.
Then for any $t_1,t_2\in I$ with $t_1<t_2$ one has
\be\label{e:hinge comparison}
\|p-\ga(t_1)\| + \|\ga(t_2)-p\|  \ge t_2-t_1 . 
\ee
\end{cor}

\begin{proof}
By continuity it suffices to prove \eqref{e:hinge comparison}
only for values of $t_1,t_2$ that are neither
reflection moments nor endpoints of~$I$.
For such $t_1,t_2$, see Lemma~\ref{l:monotone} applies
and \eqref{e:monotone} holds.
By Lemma~\ref{l:momentum bounds} we have
\[
M_{\ga,p}(t_2) \le \|\ga(t_2)-p\|
\]
and 
\[
M_{\ga,p}(t_1) \ge -\|p-\ga(t_1)\| . 
\]
These inequalities and \eqref{e:monotone} imply that
\[
t_2 - t_1 \le  M_{\ga,p}(t_2) - M_{\ga,p}(t_1) \le  \|\ga(t_2)-p\| + \|p-\ga(t_1)\| ,
\]
as claimed.
\end{proof}

\section{Proof of the Theorems} \label{sec:proofs}

We continue using the notation from the previous section.
For $p\in X$ and $r>0$ define a set $B^\pm_r(p)\subset X$ by
\be\label{e:Bpm}
 B^\pm_r(p) = \{ x\in X : \|x-p\|<r \text{ and } \|p-x\|<r \} .
\ee
In other words, $B^\pm_r(p)$ is the intersection of the balls of radius $r$ centered at $p$ for the norm $\Norm$ and the opposite norm $\Norm^-$ (see \S\ref{subsec:prelim}).

For $p\in\ov\Omega$ denote by $\Omega_p$
the complement of the union of the walls $W_i$ not containing~$p$:
\be\label{e:Omega_p}
 \Omega_p = X \setminus \bigcup_{i:p\notin W_i} W_i ,
\ee
and set
\be\label{e:r_Omega}
 r_\Omega(p) = \sup \{ r>0 : B^\pm_r(p) \subset \Omega_p\} .
\ee
Since $\Omega_p$ is an open set containing $p$, we have $r_\Omega(p)>0$.
The definitions imply that for any positive $r\le r_\Omega(p)$,
the set $B^\pm_r(p)$ intersects only those walls $W_i$ that contain $p$.

\paragraph{Proof of Theorem \ref{t:space}.}
Let $(X,\Norm,\Omega)$ be a semi-dispersing Minkowski billiard
of the form \eqref{e:walls}, $p\in\ov\Omega$ and $\ep>0$.
Define $U=B^\pm_r(p)$ where $r=\min\{\frac\ep2, r_\Omega(p)\}$,
see \eqref{e:r_Omega}.
We are going to prove that $U$ satisfies the requirements of Theorem \ref{t:space}.

Let $\ga\co I\to\ov\Omega$ be a trajectory contained in $U$. Then it is contained in $\Omega_p$, hence all reflection points belong to the walls containing $p$. By Corollary \ref{cor:hinge comparison}, for any $t_1,t_2\in I$ with $t_1<t_2$ we have
$$
 t_2-t_1 \le \|p-\ga(t_1)\| + \|\ga(t_2)-p\| < 2r \le \ep ,
$$
where the second inequality in the chain follows
from the definition of $B^\pm_r(p)$, see \eqref{e:Bpm}.
Thus, the length of the interval $I$ is no greater than~$\ep$.
\qed

\paragraph{Proof of Theorem \ref{t:time}.}
Let $(X,\Norm,\Omega)$ be a semi-dispersing Minkowski billiard
of the form \eqref{e:walls},
and let $\ga\co I\to \ov\Omega$ be a trajectory
where $I\subset\R$ is a bounded interval.
We may assume that $I$ is a closed interval of the form $I=[a,b]$.
Indeed, if $I$ is open or half-open, then by Lipschitz
continuity $\ga$ extends continuously to the closure of this interval.

Suppose that $\ga$ has infinitely many reflection moments
and let $t_0$ be an accumulation point of these moments.
The definition of a trajectory implies that $t_0=a$ or $t_0=b$.
The case $t_0=a$ reduces to the case $t_0=b$ by
reversing the time and replacing the norm by its opposite.
Thus we may assume that $b$ is an accumulation point of reflection moments.

Let $p=\ga(b)$ and choose $c\in[a,b)$ sufficiently close to $b$
so that $\ga([c,b])\subset\Omega_p$, see \eqref{e:Omega_p}.
We will henceforth consider only the part $\ga|_{[c,b]}$ of the trajectory.
On this part all reflection points belong to walls containing $p$,
so Corollary \ref{cor:hinge comparison} applies.
Substituting $t_1=c$ and $t_2=b$ into \eqref{e:hinge comparison}
and taking into account that $\ga(b)-p=0$, we obtain
\be\label{e:main1}
   \|p-\ga(c)\| \ge b - c .
\ee
On the other hand, by the subadditivity of the norm,
for all $t\in[c,b]$ we have the triangle inequality
\be\label{e:main2}
 \|\ga(t)-\ga(c)\| + \|p-\ga(t)\| \ge \|p-\ga(c)\| \ge b-c ,
\ee
where the second inequality follows from \eqref{e:main1}. Since the trajectory is parametrized with unit speed, the length of each of the segments $\ga|_{[c,t]}$ and $\ga|_{[t,b]}$ is no less than the distance between its endpoints, that is,
\[
 t - c \ge \|\ga(t)-\ga(c)\|
\]
and
\[
 b - t \ge \|p-\ga(t)\| .
\]  
Summing these two inequalities
we obtain that the triangle inequality in \eqref{e:main2} turns into equality.
By the strict convexity of the norm this
is possible only if the point $\ga(t)$ lies
on the straight line segment between $\ga(c)$ and $p$.
Therefore $\ga|_{[c,b]}$ parametrizes a straight line segment,
and this means that there are no reflections on the time interval $[c,b]$.
This contradicts the assumption that $b$ is an accumulation
point of reflection moments.
This finishes the proof of Theorem~\ref{t:time}.
\qed


\begin{bibdiv}
\begin{biblist}

\bib{AKP19}{book}{
   author={Alexander, Stephanie},
   author={Kapovitch, Vitali},
   author={Petrunin, Anton},
   title={An invitation to Alexandrov geometry: CAT(0) spaces},
   series={SpringerBriefs in Mathematics},
   publisher={Springer, Cham},
   date={2019},
   pages={xii+88},
   isbn={978-3-030-05311-6},
   isbn={978-3-030-05312-3},
   doi={10.1007/978-3-030-05312-3},
}

\bib{BFK98}{article}{
   author={Burago, D.},
   author={Ferleger, S.},
   author={Kononenko, A.},
   title={Uniform estimates on the number of collisions in semi-dispersing
   billiards},
   journal={Ann. of Math. (2)},
   volume={147},
   date={1998},
   number={3},
   pages={695--708},
   issn={0003-486X},
   doi={10.2307/120962},
}

\bib{BFK98b}{article}{
   author={Burago, D.},
   author={Ferleger, S.},
   author={Kononenko, A.},
   title={A geometric approach to semi-dispersing billiards},
   conference={
      title={Hard ball systems and the Lorentz gas},
   },
   book={
      series={Encyclopaedia Math. Sci.},
      volume={101},
      publisher={Springer, Berlin},
   },
   isbn={3-540-67620-1},
   date={2000},
   pages={9--27},
   doi={10.1007/978-3-662-04062-1\_2},
}

\bib{BI13}{article}{
   author={Burago, Dmitri},
   author={Ivanov, Sergei},
   title={Polyhedral Finsler spaces with locally unique geodesics},
   journal={Adv. Math.},
   volume={247},
   date={2013},
   pages={343--355},
   issn={0001-8708},
   doi={10.1016/j.aim.2013.07.007},
}


\bib{Bus48}{article}{
   author={Busemann, Herbert},
   title={Spaces with non-positive curvature},
   journal={Acta Math.},
   volume={80},
   date={1948},
   pages={259--310},
   issn={0001-5962},
   doi={10.1007/BF02393651},
}


\bib{Gal81}{article}{
   author={Galperin, G. A.},
   title={Systems of locally interacting and repelling particles that are
   moving in space},
   language={Russian},
   journal={Trudy Moskov. Mat. Obshch.},
   volume={43},
   date={1981},
   pages={142--196},
   issn={0134-8663},
}

\bib{GluMat25}{article}{
   author={Glutsyuk, A. A.},
   author={Matveev, V. S.},
   title={If a Minkowski billiard is projective, then it is the standard billiard},
   journal={Sbornik: Mathematics},
   year={2025},
   volume={216},
   number={5},
   pages={638--653},
   doi={10.4213/sm10182e},
}


\bib{GT02}{article}{
   author={Gutkin, Eugene},
   author={Tabachnikov, Serge},
   title={Billiards in Finsler and Minkowski geometries},
   journal={J. Geom. Phys.},
   volume={40},
   date={2002},
   number={3-4},
   pages={277--301},
   issn={0393-0440},
   doi={10.1016/S0393-0440(01)00039-0},
}

\bib{Halpern77}{article}{
   author={Halpern, Benjamin},
   title={Strange billiard tables},
   journal={Trans. Amer. Math. Soc.},
   volume={232},
   date={1977},
   pages={297--305},
   issn={0002-9947},
   doi={10.2307/1998942},
}

%

\bib{IL19}{article}{
   author={Ivanov, Sergei},
   author={Lytchak, Alexander},
   title={Rigidity of Busemann convex Finsler metrics},
   journal={Comment. Math. Helv.},
   volume={94},
   date={2019},
   number={4},
   pages={855--868},
   issn={0010-2571},
   doi={10.4171/cmh/476},
}


\bib{Sinai78}{article}{
   author={Sinai, Ja. G.},
   title={Billiard trajectories in a polyhedral angle},
   language={Russian},
   journal={Uspehi Mat. Nauk},
   volume={33},
   date={1978},
   number={1(199)},
   pages={229--230},
   issn={0042-1316},
   translation={
      journal={Russian Mathematical Surveys},
      volume={33},
      date={1978},
      number={1},
      pages={219--220},
      doi={10.1070/RM1978v033n01ABEH002254},
   },
}

\bib{ProjTab}{article}{
   author={Tabachnikov, Serge},
   title={Introducing projective billiards},
   journal={Ergodic Theory Dynam. Systems},
   volume={17},
   date={1997},
   number={4},
   pages={957--976},
   issn={0143-3857},
   doi={10.1017/S0143385797086239},
}

\bib{Vas79}{article}{
   author={Vaserstein, L. N.},
   title={On systems of particles with finite-range and/or repulsive
   interactions},
   journal={Comm. Math. Phys.},
   volume={69},
   date={1979},
   number={1},
   pages={31--56},
   issn={0010-3616},
}

\end{biblist}
\end{bibdiv}
\end{document}